\documentclass[11pt]{amsart}

\usepackage[utf8]{inputenc}
\usepackage[colorlinks]{hyperref}
\usepackage{mathtools, amsthm, amsfonts, amssymb, amsmath}
\usepackage{accents} 
\usepackage{microtype}
\hypersetup{
  linkcolor=[rgb]{0.3,0.3,0.6},
  citecolor=[rgb]{0.2, 0.6, 0.2},
  urlcolor=[rgb]{0.6, 0.2, 0.2}
}

\usepackage[all,cmtip]{xy}

\usepackage{latexsym}
\usepackage{mathtools}
\mathtoolsset{showonlyrefs,showmanualtags}
\usepackage{thmtools}
\usepackage{mathrsfs}
\usepackage{textcomp}
\usepackage{setspace}
\usepackage{enumerate}
\usepackage{upgreek}

\newcommand{\vvirg}{ , \dots , }

\newcommand{\bfJ}{\mathbf{J}}

\newcommand{\calO}{\mathcal{O}}

\newcommand{\bbC}{\mathbb{C}}

\newcommand{\bbN}{\mathbb{N}}

\newcommand{\bbP}{\mathbb{P}}

\newcommand{\bbT}{\mathbb{T}}

\newcommand{\frakS}{\mathfrak{S}}

\newcommand{\rmR}{\mathrm{R}}

\renewcommand{\phi}{\varphi}

\renewcommand{\bar}[1]{\overline{#1}}

\newcommand{\rank}{\mathrm{rank}}
  
\DeclareMathOperator{\Sym}{Sym}

\newcommand{\Sub}{\mathrm{Sub}}

\newcommand{\GL}{\mathrm{GL}}

\usepackage{hyperref}
\makeindex

\newcommand{\CB}{\mathit{CB}}

\newcommand{\ev}{\mathrm{ev}}

\newcommand{\con}{\mathrm{con}}

\newcommand{\linspan}[1]{\langle #1 \rangle}

\newtheorem{theorem}{Theorem}[section]
\newtheorem{proposition}[theorem]{Proposition}
\newtheorem{lemma}[theorem]{Lemma}
\newtheorem{corollary}[theorem]{Corollary}

\theoremstyle{definition}
\newtheorem{definition}[theorem]{Definition}

\newtheorem{remark}[theorem]{Remark}

\newcommand{\sm}{\mathrm{sm}}


\author{Luca Chiantini and Fulvio Gesmundo}

\address[L. Chiantini]{Dipartimento di Ingegneria dell'Informazione e Scienze Matematiche, Universit\`a di Siena, Italy}
\email{luca.chiantini@unisi.it}

\address[F. Gesmundo]{Saarland University, Saarbr\"ucken 66123, Germany ; 
(current) Institut de Mathématiques de Toulouse; UMR5219 -- Université de Toulouse; CNRS -- UPS, F-31062 Toulouse Cedex 9, France}
\email{fgesmund@math.univ-toulouse.fr}

\title[Cubic forms of low rank]{Decompositions and Terracini loci of \\ cubic forms of low rank}

\usepackage[top=3cm,bottom=3cm,left=3.5cm,right=3.5cm]{geometry}

\begin{document}

\begin{abstract}
 We study Waring rank decompositions for cubic forms of rank $n+2$ in $n+1$ variables. In this setting, we prove that if a concise form has more than one non-redundant decomposition of length $n+2$, then all such decompositions share at least $n-3$ elements, and the remaining elements lie in a special configuration. Following this result, we give a detailed description of the $(n+2)$-th Terracini locus of the third Veronese embedding of $n$-dimensional projective space.
 \end{abstract}

\maketitle

\section{Introduction}

For an integer $n$, let $V$ be a complex vector space of dimension $n+1$ and let $\bbP^n = \bbP V$ be the projective space of lines in $V$. Identify $V$ with the space of linear forms on $V^*$ and let $S^d V$ be the space of homogeneous polynomials of degree $d$ on $V^*$. The Waring rank of a form $F \in S^d V$ is 
\[
\rmR(F) = \min \{ r : F = L_1^d + \cdots + L_r^d \text{ for some $L_1 \vvirg L_r \in V$}\}. 
\]
The $d$-th Veronese embedding of $\bbP V$ is the map $v_d : \bbP V \to \bbP S^d V$ defined by $v_d([L]) = [L^d]$. If $A$ is a set of $r$ points in $\bbP V$, we say that $A$ has length $r$, and we write $\ell(A) = r$. Geometrically, the Waring rank of $F$ can be equivalently defined as 
\[
\rmR(F) = \min \{ r : [F] \in \langle v_d(A) \rangle \text{ for a set of points $A \subseteq \bbP V$ with $\ell(A) = r$}\}.
\]
Here $\langle - \rangle$ denotes the (projective) linear span. A finite set $A$ with $\ell(A) = r$ such that $[F] \in \linspan{\nu_d(A)}$ is a \emph{decomposition} of $F$ of length $r$.  If $A = \{ [L_1] \vvirg [L_r]\}$, this is equivalent to the existence of an expression 
\begin{equation}\label{eqn: decomposition F}
 F = \alpha_1 L_1^d + \cdots + \alpha_r L_r^d
\end{equation}
for some $\alpha_1 \vvirg \alpha_r \in \bbC$. We say that the decomposition is \emph{non-redundant} (or irredundant) if there is no proper subset $A' \subsetneq A$ such that $[F] \in \linspan{v_d(A')}$; this is equivalent to saying that $L_1^d \vvirg L_r^d$ are linearly independent and no coefficient $\alpha_i$ in \eqref{eqn: decomposition F} is $0$. We say that the decomposition is \emph{minimal} if its length is the minimal possible, namely $r = \rmR(F)$. If $F$ is a form of rank $r$ which has a unique minimal decomposition of length $r$, we say that $F$ is identifiable.

We say that $F \in S^d V$ is concise if there is no proper subspace $V' \subsetneq V$ such that $F \in S^d V'$; equivalently, there is no change of coordinates in $V$ after which $F$ can be written in fewer than $n+1$ variables. It is a classical fact, implicitly appearing in \cite{Sylv:PrinciplesCalculusForms}, that $F$ is concise if and only if its first order partial derivatives are linearly independent; equivalently $F$ is concise if and only if its partial derivatives of order $d-1$ span $V$; we refer to \cite{IarrKan:PowerSumsBook,Car:ReducingNumberVariables} for a modern discussion on this topic. If $F$ is concise, then every decomposition $A$ of $F$ satisfies $\langle A \rangle = \bbP V$; in other words, no decomposition of $F$ is contained in a hyperplane. In particular $\rmR(F) \geq n+1$.  A partial converse of this property in low rank is given in \autoref{lemma: minimal plus one concise}.

For a subset $A \subseteq \bbP V$ with $\ell(A) = r$, we say that $A$ is linearly independent if $\dim \langle A \rangle = \ell(A) - 1$, which is the maximum possible value; here $\dim$ is the dimension in projective space. The Kruskal rank of $A$ is the largest integer $k$ such that every subset of $k$ elements of $A$ is linearly independent. We say that $A$ is in \emph{linear general position} (LGP) if any subset of $n+1$ elements of $A$ is linearly independent, namely the Kruskal rank of $A$ is $\min \{ \ell(A) , n+1\}$, which is the maximum possible value.

In this work, we study non-redundant decompositions of length $n+2$ for concise cubic forms in $n+1$ variables. This is the first case where Kruskal's criterion does not guarantee uniqueness of the decomposition and in fact there are examples of forms having multiple decompositions, see e.g. \cite{Derk:KruskalSharp}. However, our main result shows that if the form is concise, then non-uniqueness of a decomposition can only occur under very special conditions, and in particular very low Kruskal rank.

\begin{theorem}\label{thm: main theorem}
Let $V$ be a vector space of dimension $n+1$, with $n \geq 1$. Let $F \in S^3 V$ be concise with $\rmR(F) \leq n+2$. Then one of the following occurs:
 \begin{enumerate}[(I)]
  \item there is a unique non-redundant decomposition $A$ of length $n+2$ for $F$; in this case the Kruskal rank of $A$ is at least $4$;
  \item there are infinitely many non-redundant decompositions of $F$ of length $n+2$; any two decompositions $A,B$ of length $n+2$ intersect in at least $n-3$ points; moreover $(A \setminus B) \cup (B \setminus A)$ lies in the union of two lines; in particular the Kruskal rank of any decomposition is $2$;
  \item there are infinitely many non-redundant decompositions of $F$ of length $n+2$; any two decompositions $A,B$ of length $n+2$ intersect in at least $n-2$ points; moreover $(A \setminus B) \cup (B \setminus A)$ is contained in the union of two planes; in particular the Kruskal rank of any decomposition is at most $3$.
 \end{enumerate}
\end{theorem}

\autoref{thm: main theorem} is built on an induction argument for $n \geq 4$. The base of the induction is given by the cases $n=1,2,3$, where special configurations of points yield some pathological behaviour. These special configurations already appear in the literature \cite{ComaSigu:RankBinaryForms,BerGimIda:ComputingSymmetricRankSymmetricTensors,BalBer:StratificationFourthSecantVeronese,Ball:StratificationSigma5} and in some cases they already appear implicitly in the classical \cite{Sylv:PrinciplesCalculusForms,Terr:seganti}. The general cases are obtained as extensions of a special configuration on a subspace; this result is original to the best of our knowledge.

As a consequence of \autoref{thm: main theorem}, in \autoref{sec:terraciniloci}, we obtain a detailed description of the $(n+2)$-th Terracini locus of the Veronese variety $v_3(\bbP^n)$ in the concise setting. 

The study of identifiability of homogeneous polynomials, and more generally of tensors, is a classical topic in algebraic geometry and invariant theory. Generic forms in $S^d V$ are identifiable only for few combination of degree $d$ and number of variables $n+1$, namely $(d,n) = (2k-1,1), (3,3), (5,2)$: identifiability in these cases was known classically \cite{Sylv:PrinciplesCalculusForms,Hilblet}; the full proof that these are the only cases where generic identifiability holds is recent \cite{GalMel:IdentifiabilityHomPoly}. On the other hand, generic forms of subgeneric rank are identifiable with only few exceptions \cite{BocChiOtt:Refined_identifiability_tensors,ChiOttVan:GenericIdentifSubgenRank}. In general, however, little is known about the explicit genericity condition providing identifiability. The most general tool in this context is Kruskal's criterion \cite{Krusk:ThreeWayArrays}, which guarantees uniqueness of decomposition under the hypothesis of large enough \emph{Kruskal rank}; a variant of Kruskal criterion is given in \cite{LovPet:GeneralizationKruskal}. A geometric study of the subvarieties of non-identifiable forms, in special cases, is given in  \cite{AngChiVan:IdentifiabilityBeyondKruskal,AngChi:IdentifiabilityTernaryForms,ChiOtt:FootnoteFootnote}, which are based on properties of the Hilbert function of points \cite{Chi:HilbertFunctionsTensorAn} and on apolarity theory \cite{IarrKan:PowerSumsBook}.

\subsection*{Acknowledgements} This work is partially supported by  the Thematic Research Programme ``Tensors: geometry, complexity and quantum entanglement'', University of Warsaw, Excellence Initiative -- Research University and the Simons Foundation Award No. 663281 granted to the Institute of Mathematics of the Polish Academy of Sciences for the years 2021-2023. We thank IMPAN and the AGATES organizers for providing an excellent research environment throughout the semester.

\section {Preliminaries}
 
Let $X \subseteq \bbP^N$ be an irreducible projective variety. The $r$-th secant variety of $X$ is 
\[
 \sigma_r(X) = \bar{\{ p \in \bbP^N : p \in \langle x_1 \vvirg x_r \rangle \text{ for some } x_1 \vvirg x_r \in X  \}};
\]
here the closure can be taken equivalently in the Zariski or the Euclidean topology. The $r$-th symmetric product of $X$ is $X^{(r)} = X^{\times r} / \frakS_r$, where the symmetric group $\frakS_r$ acts by permuting the factors of $X^{\times r}$. One can verify that $X^{(r)}$ is a projective variety. The $r$-th abstract secant variety of $X$ is 
\[
 A\sigma_r(X) = \bar{\{ ( (x_1 \vvirg x_r), p) : p \in \langle x_1 \vvirg x_r \rangle \}} \subseteq X^{(r)} \times \bbP^N.
\]
There are two natural projections $\pi_X : A\sigma_r(X) \to X^{(r)}$, which is surjective, and $\pi_\sigma : A\sigma_r(X) \to \bbP^N$, which surjects onto the secant variety $\sigma_r(X)$. It is easy to verify that $ A\sigma_r(X)$ is irreducible and $\dim A\sigma_r(X) = r \dim X + (r-1)$; in fact, $A\sigma_r(X)$ is birational to a projective bundle over $X^{(r)}$ whose fibers are copies of $\bbP^{r-1}$. We say that $\sigma_r(X)$ is non-defective, or that $X$ is not $r$-defective, if $\dim \sigma_r(X) = \min \{ N , r \dim X + (r-1)\}$: if $\sigma_r(X) \subsetneq \bbP^N$, this is equivalent to the condition that $\pi_\sigma$ has generically finite fibers. 

\begin{definition} Let $X  \subseteq \bbP^N$ be an irreducible variety and let $r \geq 1$ be an integer such that $N \geq r \dim  X + (r-1)$. The $r$-th Terracini locus of $X$ is 
\[
\bbT_r(X) = \bar{\left\{ (x_1 \vvirg x_r) : \begin{array}{l} x_i \in X^{\sm} \text{ are linearly independent}, \\ 
                                        T_{x_1} X , \cdots , T_{x_r} X \text{ are not linearly independent}
                                       \end{array}
\right\}} \subseteq X^{(r)},
\]
where $X^{\sm}$ is the open set of the smooth points of $X$ and $T_xX$ denotes the affine tangent space to $X$ at $x$.
\end{definition}
This object was introduced in \cite{BalChi:TerraciniLocus,BalBerSan:TerraciniLocusThreePts} as a tool to study identifiability of tensors and singularities of secant varieties. Terracini's Lemma \cite[Lemma 1]{BCCGO:HitchhikerGuide} guarantees that if $\sigma_r(X)$ is non-defective and it does not fill the space, then $\bbT_r(X)$ is a proper subvariety of $X^{(r)}$.

The preimage $\pi_X^{-1} ( \bbT_r(X)) \subseteq A\sigma_r(X)$ is contained in the locus where the differential of $\pi_\sigma$ drops rank. If a point $q \in \bbP^N$ is in the image $\pi_\sigma \pi_X^{-1}(\bbT_r(X)) \subseteq \sigma_r(X)$ then either $\pi_\sigma^{-1}(q)$ is positive dimensional, or $q$ is a cuspidal singularity of $\sigma_r(X)$. 

By definition, the Terracini locus is the sets of points $\{x_1 \vvirg x_r\}$ with the property that they impose independent conditions on $\calO_X(1)$ but the double points $\{ 2x_1 \vvirg 2x_r\}$ do not impose independent conditions on $\calO_X(1)$. In fact, with this formulation, one can give the definition of Terracini locus for an abstract variety and any (ample) line bundle, see \cite{BalChi:TerraciniLocus,BalVen:NoteTerracini}.

An important tool in the study of sets of point in projective space is given by the Hilbert function. Let $\bbC[V] = \Sym(V^*)$ be the homogeneous coordinate ring of $\bbP V$, which is a polynomial ring in $n+1$ variables. If $Z \subseteq \bbP V$ is a subvariety with homogeneous ideal $I(Z)$, the Hilbert function of $Z$ is 
\begin{align*}
 h_Z : \bbN &\to \bbN \\ 
 t &\mapsto \dim ( \bbC[V]_t  / I(Z)_t).  
\end{align*}
If $Z$ is a set of points, $Z = \{ [v_1] \vvirg [v_r]\}$ for $v_j \in V$, let $\ev_i : \bbC[V] \to \bbC$ be the evaluation map at $v_i$. For every $t$, the restriction $\ev_i : \bbC[V]_t \to \bbC$ is linear, and the Hilbert function of $Z$ is characterized as 
\[
 h_Z(t) = \rank( (\ev_1 \vvirg \ev_r) : \bbC[V] \to \bbC^r).
\]
The \emph{first difference} of the Hilbert function of $Z$ is $Dh_Z(t) = h_Z(t) - h_Z(t-1)$. It is often identified simply with the sequence of its non-zero values, which is called the h-vector of $Z$. We record three main properties of the Hilbert function and the h-vector, which will be useful in the following. We refer to \cite{Chi:HilbertFunctionsTensorAn} for details and proofs.

\begin{proposition}\label{prop: basics HF}
 Let $Z' \subseteq Z \subseteq \bbP V$ be finite sets of points. Then 
 \begin{enumerate}[(i)]
 \item $Dh_{Z'}(t) \leq Dh_{Z}(t)$ for every $t$; in particular $h_{Z'}(t) \leq h_{Z}(t)$ for every $t$.
  \item There is an integer $\tau \geq 0$ such that $h_Z(t)$ is strictly increasing for $t \leq \tau$ and $h_Z(t) = \ell(Z)$ for $t \geq \tau$. In particular $Dh_Z (t) > 0$ if $t\leq \tau$ and $Dh_Z(t) = 0$ if $t > \tau$. 
 \item If $Dh_Z(t) \leq t$, then $Dh_Z(t+1) \leq Dh_Z(t)$.
 \end{enumerate}
\end{proposition}

Understanding what functions $h$ can occur as Hilbert functions of sets of points is the object of a long line of research. Macaulay characterized the \emph{maximal growth} of $h_Z$. In \cite{BigGerMig:GeometricConsequencesMacaulay}, strong consequences of Macaulay's result are given; we state here a restricted version of \cite[Theorem 3.6]{BigGerMig:GeometricConsequencesMacaulay}, extending a result from \cite{Davis:CompleteIntersections} in the case of $\bbP^2$.
\begin{theorem}\label{thm: BGM}
 Let $Z \subseteq \bbP V$ be a finite set of points. If $s := Dh_Z(t_0) = Dh_Z(t_0+1) \leq t_0$ for some $t_0$ then there exists a reduced curve $C \subseteq \bbP V$ of degree $s$ such that, setting $Z' = Z \cap C$, one has 
 \begin{itemize}
  \item $Dh_{Z'}(t) = Dh_C(t)$ for $t \leq t_0+1$;
  \item $Dh_{Z'}(t) = Dh_Z(t)$ for $t \geq t_0$.
 \end{itemize}
\end{theorem}
Intuitively, \autoref{thm: BGM} says that if $Dh_Z$ is constant and small on an interval, then a subset of points of $Z$ lies on a low degree curve. This is useful to deduce pathologies of certain sets of points and prove that under suitable genericity assumptions only certain Hilbert functions are possible.

Sets of points defining non-redundant  decompositions satisfy particular conditions that are reflected in their Hilbert function. An important one concerns the Cayley-Bacharach property. A finite set $Z \subseteq \bbP V$ satisfies the Cayley-Bacharach property in degree $t$, denoted $\CB(t)$, if, for every $p \in Z$, every form of degree $d$ vanishing on $Z \setminus \{p\}$ vanishes at $p$ as well; in other words, $Z$ satisfy $\CB(t)$ if for every $p \in Z$, $I(Z \setminus \{p\})_t = I(Z)_t$.

The role of the Cayley-Bacharach property in the study of decompositions of homogeneous polynomials is stated in the following result.
\begin{proposition}[\cite{AngChi:IdentifiabilityTernaryForms}, Proposition 2.25]\label{prop: CB for nonredundant}
 Let $F \in S^d V$ and let $A,B \subseteq \bbP V$ be non-redundant disjoint decompositions of $F$. The $A \cup B$ satisfies the Cayley-Bacharach property in degree $d$.
\end{proposition}
In particular, the failure to satisfy the Cayley-Bacharach property allows one to rule out the existence of certain decompositions. This will be achieved using the fact that the Cayley-Bacharach property poses strong conditions on the Hilbert function of a set of points. More precisely, the following holds, see \cite[Thm. 4.9]{AngChiVan:IdentifiabilityBeyondKruskal}.
\begin{theorem}\label{thm: CB implies HF}
 Let $Z \subseteq \bbP V$ be a set of points satisfying $\CB(t)$. Then, for every $s \leq t+1$
 \[
  Dh_Z(0) + \cdots + Dh_Z(s) \leq Dh_Z(t+1) + \cdots + Dh_Z(t+1-s).
 \]
\end{theorem}

\subsection{First results}

In this section we prove some simple technical results toward the proof \autoref{thm: main theorem}. The first immediate observation is that $n+2$ points in LGP in $\bbP V$ can be normalized via the action of $\GL(V)$. We record a slightly more general fact in the following result, see e.g. \cite[Ch. 10]{Harris:AlgGeo}.
\begin{lemma}\label{lemma: orbits}
Let $V$ be a vector space with $\dim V = n+1$ and let $e_0 \vvirg e_n$ be a basis of $V$. Let
\[
 \Omega = \{ A \subseteq \bbP V: \ell(A) = n+2, \langle A \rangle = \bbP V\} \subseteq (\bbP V)^{(n+2)}.
\]
Then $\Omega$ is Zariski open in $(\bbP V)^{(n+2)}$; the action of $\GL(V)$ on $\Omega$ has exactly $n$ orbits $K_2 \vvirg K_{n+1}$. The elements of $K_r$ are subsets having Kruskal rank exactly $r$ and they are all equivalent to
 \[
  A_r = \{ [e_0] \vvirg [e_n] , [e_0 + \cdots + e_{r-1}]\}.
 \]
Moreover, $K_r \subseteq \bar{K_{r+1}}$ and for every $r$, $\bar{K_r}$ is irreducible of dimension $n(n+1) + r-1$.
\end{lemma}
\begin{proof}
The first part of the statement is classical. If $A \in (\bbP V)^{(n+2)}$ is a set of points having Kruskal rank $r$ and such that $\langle A \rangle = \bbP V$, then we may assume $A = \{ [v_0] \vvirg [v_n], [v_{n+1}]\}$ with $v_0 \vvirg v_n$ linearly independent and $v_{n+1} \in \langle v_0 \vvirg v_{r-1}\rangle$. Since $v_0 \vvirg v_n$ are linearly independent, there is an element $g \in \GL(V)$ such that $g v_i = e_i$ for $i = 0 \vvirg n$; as a result $g v_{n+1} = \lambda_0 e_0 + \cdots + \lambda_{r-1} e_{r-1}$ for some $\lambda_j \in \bbC \setminus \{ 0 \}$. Define an element $h \in \GL(V)$ by $h e_i = \lambda_i^{-1} e_i$ for $i \leq r-1$ and $h e_i = e_i$ for $i \geq r$. We conclude that $hg A = A_r$.

The condition that $K_r \subseteq \bar{K_{r+1}}$ is clear; moreover $\bar{K_r}$ is irreducible, because it is the closure of an orbit under the action of $\GL(V)$. Finally, to compute the dimension of $\bar{K_r}$, consider the diagram
\[
\xymatrix{
& (\bbP V)^{ n+2} \ar[dl]_{\pi_\frakS} \ar[dr]^{\pi_{n+2}} & \\
(\bbP V)^{(n+2)} & & (\bbP V)^{n+1}
}
\]
where $\pi_\frakS$ is the projection onto the quotient $(\bbP V)^{(n+2)} = (\bbP V)^{ n+2}/\frakS_{n+2}$ and $\pi_{n+2}$ is the projection onto the first $n+1$ factors. The preimage $\pi_\frakS (K_r)$ is the union of several irreducible components, all isomorphic and all surjecting onto $K_r$; one of these components is 
\[
J_r = \left \{ ([v_0] \vvirg [v_{n+1}]) \in (\bbP V)^{ n+2} :  
\begin{array}{l}
\bbP V = \langle [v_0] \vvirg [v_n] \rangle \\
v_{n+1} \in \langle v_0 \vvirg v_{r-1}\rangle
\end{array}
\right\}
\]
Since $\pi_\frakS$ is generically finite, $\dim K_r = \dim J_r$. The restriction $\pi_{n+2} : J_r \to (\bbP V)^{n+1}$ is dominant; the fiber over a generic element $([v_0] \vvirg [v_n])$ is the linear span $\bbP \langle v_0 \vvirg v_{r-1}\rangle$ which has dimension $r-1$. We conclude
\[
 \dim K_r = \dim J_r = \dim ((\bbP V)^{n+1}) + \dim \bbP^{r-1}(r-1) = n(n+1) + r-1.
\]
This concludes the proof.
\end{proof}

We observe that forms admitting non-redundant decompositions of length $n+2$ are necessarily concise:
\begin{lemma}\label{lemma: minimal plus one concise}
 Let $F \in S^d V$ be a homogeneous form with $\dim V = n+1$. Let $A \subseteq \bbP V$ be a set of $n+2$ points with $\langle A \rangle = \bbP V$. If $A$ is a non-redundant decomposition of $F$ then $F$ is concise.
\end{lemma}
\begin{proof}
By \autoref{lemma: orbits}, we may assume $A \in K_r$ for some $r$ and we can easily reduce to the case $r= n$. Write $L = x_0 + \cdots + x_n$ and assume 
\[
 F = \alpha_0 x_0^d + \cdots + \alpha_n x_n^d + \alpha_{n+1} L^d
 \]
where $\alpha_j \neq 0$ for every $j$.

Let $H$ be the space of $(d-1)$-th order partial derivatives of $F$. Up to scaling, we have 
\[
 \frac{\partial^{d-1}}{\partial x_0^{d-2}\partial x_1} F = \alpha_{n+1} L;
\]
hence $L \in H$. Moreover, $\frac{\partial^{d-1}}{\partial x_j^d} F = \alpha_j x_j + \alpha_{n+1} L$, showing $x_j \in H$ for every $j$. This shows $H = V$, therefore $F$ is concise.
\end{proof}

Similarly, forms admitting non-redundant decompositions of length exactly $n+1$ are concise, and their decomposition is unique. The uniqueness of the decomposition can be obtained via an easy calculation with the Hilbert function, but it is also a consequence of the classical Kruskal criterion for tensors, \cite{Krusk:ThreeWayArrays}.
\begin{lemma}\label{cor: uniqueness for LI decomp}
 Let $F \in S^d V$ be a homogeneous form, with $\dim V = n+1$ with $d \geq 3$. Let $A \subseteq \bbP V$ be a non-redundant decomposition of $F$ with $A$ linearly independent in $\bbP V$ and $\ell(A) = n+1$. Then $F$ is concise, $\rmR(F) = n+1$ and $A$ is the unique minimal decomposition of $F$.
 \end{lemma}
\begin{proof}
Since $A$ is linearly independent with $\ell(A) = n+1$, we may normalize it via the action of $\GL(V)$ and assume $A = \{ [x_0] \vvirg [x_n]\}$. Therefore $\langle v_d(A) \rangle = \langle [x_0^d] \vvirg [x_n^d]\rangle$ and $F = x_0^d + \cdots + x_n^d$. This shows that $F$ is concise with $\rmR(F) = n+1$.

The uniqueness of the decomposition follows by Kruskal's criterion, see \cite{Krusk:ThreeWayArrays}. The proof follows from the immediate fact that if $A \subseteq \bbP V$ is a set of linearly elements with $\ell(A) = n+1$, then $A$ is in LGP and in particular it has maximal Kruskal ranks. 
\end{proof}

A stronger result holds for forms admitting two non-redundant decompositions of length $n+1$ and $n+2$, respectively.
\begin{lemma}\label{lemma: minimal plus one for fermat}
 Let $F \in S^d V$ be a concise form with $d \geq 3$. Let $A \subseteq \bbP V$, $B \subseteq \bbP V$ be non-redundant decompositions of $F$ of length $n+1$ and $n+2$ respectively. Then $d = 3$, $\ell(A \cap B) \geq n-1$ and $(A \setminus B) \cup (B \setminus A)$ is collinear. 
\end{lemma}
\begin{proof}
We proceed by induction on $n$. If $n = 1$ the statement is clearly true. Since $5$ points on a rational normal curve of degree $d \geq 4$ are linearly independent, we immediately have $d=3$. In this case, every concise $F \in S^3 V$ has at most one decomposition of length $2$ and a two-parameter family of decompositions of length $3$. All decompositions are collinear because $\bbP V$ is a line.

Assume $n \geq 2$. First assume $A \cap B = \emptyset$. Let $Z = A \cup B$. Since $F$ is concise, we have $Dh_A = (1,n)$, which implies $Dh_Z(0) + Dh_Z(1) = n+1$; by the Cayley-Bacharach property of $Z$, we have $Dh_Z(d+1) + Dh_Z(d) \geq n+1$. This implies $Dh_Z(2) \leq 1$, which arithmetically is a contradiction by \autoref{prop: basics HF}.
 
 Therefore $A \cap B \neq \emptyset$. Without loss of generality, assume $A = \{ [x_0] \vvirg [x_n]\}$ and $F  = x_0^d + \cdots + x_n^d$. Since $A \cap B \neq \emptyset$, assume $[x_n] \in B$ and write $B = \{ [L_0] \vvirg [L_n],[x_n]\}$, so that $F = L_0^d + \cdots + L_n^d + \beta x_n^d$ for some nonzero $\beta \in \bbC$. Consider 
 \begin{equation}\label{eqn: two decomps F'}
  F' = F - \beta x_n^d = x_0^d + \cdots + x_{n-1}^d + (1-\beta) x_n^d = L_0^d + \cdots + L_n^d.
 \end{equation}
Let $B' = \{ [L_0] \vvirg [L_n]\}$; clearly $[x_n] \notin B'$. If $\beta \neq 1$, then $F'$ is concise and \eqref{eqn: two decomps F'} gives two non-redundant decompositions $F'$ of length $n+1$. By \autoref{cor: uniqueness for LI decomp}, the two decompositions are the same, in contradiction with the fact that $[x_n] \notin B'$. Therefore $\beta = 1$ and $F'$ is non-concise. Moreover, if $B'$ was linearly independent in $\bbP V$, \autoref{cor: uniqueness for LI decomp} would imply that $F'$ is concise; hence $B'$ is not linearly independent. 

Let $V' = \langle x_0 \vvirg x_{n-1} \rangle$. Notice $V' = \langle B'\rangle$ and $F'$ is concise in $S^d V'$. Let $A' = \{ [x_0] \vvirg [x_{n-1}]\}$. Now $F'$ is concise and it has two decompositions $A',B'$ of length $n$ and $n+1$ respectively. By the inductive hypothesis, we deduce $d = 3$, $\ell(A' \cap B') \geq n-2$ and $B'$ has three collinear points. Since $A = A' \cup \{ [x_n]\}$ and similarly for $B$, we deduce that the statement holds for $A$ and $B$ as well. This concludes the proof.
\end{proof}

In the study of identifiability of forms, it is often useful to consider the case where two potential decompositions of a form are disjoint. We illustrate a general method to reduce the case of non-disjoint decompositions to the one of disjoint decompositions: 
\begin{remark}\label{rmk: pass to disjoint}
Let $F \in S^d V$ and suppose $A,B \subseteq \bbP V$ give two non-redundant decompositions of $F$. Write $A = \{ [L_1] \vvirg [L_r]\}$ and $B = \{ [M_1] \vvirg [M_s]\}$. Then 
\[
 F = \alpha_1 L_1^d + \cdots + \alpha_r L_r^d = \beta_1 M_1^d + \cdots + \beta_s M_s^d.
\]
After possibly reordering the elements of $A$ and $B$, assume $L_i = M_i$ for $i = 1 \vvirg p$, where $ p = \ell(A\cap B)$. In this case, define 
\[
 F' = F - (\alpha_1 L_1^d + \cdots + \alpha_p L_p^d) = (\beta_1 - \alpha_1)M_1^d + \cdots +  (\beta_p - \alpha_p)M_p^d + \beta_{p+1} M_{p+1} + \cdots + \beta_s M_s^d;
\]
note
\[
 F ' = \alpha_{p+1} L_{p+1}^d + \cdots + \alpha_r L_r^d.
\]
Let $A' = A \setminus B$; let $B' = \{ [M_i] : \alpha _i \neq \beta_i , i = 1 \vvirg p\} \cup \{[M_{p+1}] \vvirg [M_s]\}$; then $A',B'$ are two disjoint non-redundant decompositions of $F'$. Moreover, if $A$ is a minimal decomposition for $F$, then $A'$ is a minimal decomposition for $F'$.
\end{remark}

We conclude this section with a classical result for which we do not know an explicit reference, in the form that we need in the following. In the case $n=1$, the result dates back to \cite{Sylv:PrinciplesCalculusForms}; analogous statements, in a more general setting, are discussed in \cite[Section 7.1]{BalBerChrGes:PartiallySymRkW}.

\begin{proposition}\label{prop: sylvester's d+2} Let $F$ be a concise form in $S^d V$, with $\dim V = n+1$. Let $A,B$ be non-redundant decompositions of $F$ of length $r$ and $s$ respectively. Then $r+s \geq d+2n$. 
\end{proposition}
\begin{proof} Let $p$ be the length of $A\cap B$. If $p=0$, set $F'=F$. If $p>0$, consider the form $F'$ with disjoint decompositions $A',B'$ of length $r',s'$ respectively, as constructed in \autoref{rmk: pass to disjoint}. Notice that  $r-r'=p\geq s-s'$; specifically $s-s'$ is the number of elements of $A \cap B$ appearing with the same coefficient in the two decompositions of $F$. Observe $\dim \langle B' \rangle \geq n - (s-s')$: indeed $F$ is concise, which implies that $B$ spans $\bbP V$; since $\dim \langle B \setminus B' \rangle \leq s-s'-1$, we deduce $\dim \langle B' \rangle \geq n - (s-s')$. 

Set $Z = A'\cup B'$ and let $n' = \dim \langle Z' \rangle$. We have $n' \geq \dim \langle B' \rangle \geq n-(s-s')$, so $n-n' \leq s-s' \leq r-r'$. Since $Z$ is union of two disjoint decompositions of a form of degree $d$, we have $Dh_{Z'}(d+1) \geq 1$, which implies $Dh_{Z'}(t) \geq 1$ for $t \leq d+1$. Moreover, $Dh_{Z'}(1) = \dim \langle Z' \rangle = n'$. Finally, \autoref{prop: CB for nonredundant} implies that the Cayley-Bacharach property holds for $Z'$ in degree $d$, and \autoref{thm: CB implies HF} implies $Dh_{Z'}(d+1) + Dh_{Z'}(d) \geq Dh_{Z'}(0) + Dh_{Z'}(1) \geq 1+n'$.

We deduce $Dh_{Z'} = (1,n', \delta_2 \vvirg \delta_{d+1}, \ldots)$ with $\delta_t \geq 1$ for $t \leq d+1$ and $\delta_{d+1} + \delta_d \geq 1+n'$. We conclude $r' + s' = \ell(Z') \geq \sum_{t=0}^{d+1} \delta_t \geq 2+2n' + (d-2) = d + 2n'$. Therefore
\[
r+s=(r-r')+(s-s')+(r'+s')\geq (n-n')+(n-n')+d+2n' = d + 2n 
\]
as desired.
\end{proof}

The bound of \autoref{prop: sylvester's d+2} is sharp. An example where the bound is attained can be constructed, for every $n$ and $d$, as follows. When $d=2$, it is a classical fact \cite{Terr:seganti} that general forms of maximal rank have many decompositions of length $n+1$, which attains the bound. 

For higher values of $d$, write $V = V' \oplus V''$ where with $\dim V' = 2$, $\dim V' = n-1$. Let let $F' \in S^d V'$ be a binary form having two non-redundant disjoint decompositions $A',B'$ of length $r',s'$ with $r'+s'=d+2$; it is a classical result that such forms exist \cite{Sylv:PrinciplesCalculusForms,ComaSigu:RankBinaryForms}. Let $F'' = L_1^d+\dots +L_{n-1}^d$ be a sum of $n-1$ $d$-powers of generic linear forms. Then $F = F' + F''$ is concise and it has two distinct non-redundant decompositions $A=A'\cup\{L_1,\dots,L_{n-1}\}$ and $B=B'\cup   \{L_1,\dots,L_{n-1}\}$ with $\ell(A) = r' + (n-1)$, $\ell(B) = s' + (n-1)$ and we have $\ell(A) + \ell(B) = r'+s'+2(n-1) = d+2n$.

\section{First cases}
In this section, we characterize decompositions of length $n+2$ in $\bbP^n$ for $n =1,2$.
\subsection*{Case $n=1$} ~ 

This case is classical, see e.g. \cite{ComaSigu:RankBinaryForms,BerGimIda:ComputingSymmetricRankSymmetricTensors}. Let $V$ be a vector space with $\dim V= 2$. The generic Waring rank in $\bbP S^3 V$ is $2$. In this case \autoref{thm: main theorem} is easily verified. Let $F \in S^3 V$ be a concise form. If $\rmR(F) = 2$, then $F$ has a unique decomposition of length $2$ by Kruskal criterion, and in particular \autoref{cor: uniqueness for LI decomp}; moreover $F$ has a $2$-dimensional family of decompositions of length $3$. If $\rmR(F) = 3$, then $[F]$ is an element of $\tau(v_3(\bbP^1)) \setminus v_3(\bbP^1)$, where $\tau(v_3(\bbP^1)) = \{ [L_1^2 L_2] \in \bbP S^3 V: [L_1],[L_2] \in \bbP^1\}$; in this case $F$ has a two dimensional family of decompositions of length $3$. 

Since $\bbP V= \bbP^1$, all decompositions are defined by collinear points. In particular, both case (II) and case (III) of \autoref{thm: main theorem} trivially hold.
\medskip

\subsection*{Case $n=2$}~ 

Let $V$ be a vector space with $\dim V= 3$. The generic Waring rank in $\bbP S^3 V$ is $4$. Let $F \in S^3 V$ be a concise form. It is known that $3 \leq \rmR(F) \leq 5$ \cite{LanTei:RanksBorderRanksSymTensors} and we are interested in the dense subset of forms satisfying $\rmR(F) \leq 4$. Notice $\dim A\sigma_4(v_3(\bbP^2)) = 12$. Therefore the projection map $A\sigma_4(v_3(\bbP^2)) \to \bbP S^3 V$ has generic fibers of dimension $2$. This guarantees that if $\rmR(F) \leq 4$, then $F$ has at least a $2$-dimensional family of decompositions of length $4$.

In particular, case (III) of \autoref{thm: main theorem} trivially holds. We provide an explicit geometric characterization of the decompositions, which implicitly appears in the proof of \cite[Thm. 5]{BerGimIda:ComputingSymmetricRankSymmetricTensors}.
\begin{theorem}\label{thm: n3}
 Let $F \in S^3 V$ be a concise form with $\dim V = 3$. Let $A,B$ be non-redundant decompositions of $F$, with $\ell(A) = \ell(B) = 4$. Then one of the following holds:
 \begin{itemize}
  \item $A \cap B = \emptyset$ and there is a conic passing through $A \cup B$; if $A$ has three collinear points, so does $B$.
  \item $\ell(A \cap B) \geq 1$ and all non-redundant decompositions of $F$ of length $4$ have three collinear points, on a line $\Lambda$.
 \end{itemize}
\end{theorem}
\begin{proof}
 The statement follows from \autoref{prop: unique conic} and \autoref{prop: three collinear on plane} below. \autoref{prop: unique conic} analyzes the case where $A$ and $B$ are disjoint and \autoref{prop: three collinear on plane} the one where they are not.
\end{proof}

\begin{proposition}\label{prop: unique conic}
 Let $F \in S^3 V$ be a concise form with $\dim V = 3$. Let $A,B$ be non-redundant decompositions of $F$, with $\ell(A) = \ell(B) = 4$. If $A$ and $B$ are disjoint then there is a unique conic $C$ vanishing on $A \cup B$. Moreover, if $\Lambda$ is a line with $\ell(A \cap \Lambda) \geq 3$, then $C = \Lambda \cup \Lambda'$ for a line $\Lambda'$ and $\ell(B \cap \Lambda') =3$.
\end{proposition}
\begin{proof}
Since $F$ is concise, there are no linear forms vanishing on $A$. Therefore the h-vector of $A$ is $Dh_A = (1,2,1)$ and there is a pencil of conics vanishing on $A$.

Let $Z = A \cup B$; since $A \cap B =\emptyset$, $\ell(Z) = 8$. Since $[F] \in \linspan {v_3(A)}\cap \linspan {v_3(B)}$, we have that $v_3(Z)$ is not linearly independent, namely $\dim \linspan{v_3(Z)} < \ell(Z) - 1 = 7$. In particular, $Dh_Z(4)>0$. Since $Dh_A = (1,2,1)$, we have $Dh_Z(0) + Dh_Z(1) \geq 3$. Since $A,B$ are non-redundant, they satisfy the Cayley-Bacharach property in degree $3$ by \autoref{prop: CB for nonredundant}: therefore \autoref{thm: CB implies HF} provides $Dh_Z(3)+Dh_Z(4)\geq 3$, so $Dh_Z(2) \leq 2$. If $Dh_Z(2) \leq 1$, by \autoref{prop: basics HF}, we have $Dh_Z(i) \leq 1$ for $i\geq 2$, which provides a contradiction; therefore $Dh_Z(2) = 2$. We obtain $Dh_Z = (1,2,2,2,1)$. In particular, $h_Z(2) = 1+2+2 = 5$, which implies that there is a unique conic vanishing on $Z$.

Now suppose $A$ contains three collinear points, lying on a line $\Lambda \subseteq \bbP V$. Since $A \subseteq C$, we deduce that the conic $C$ passing through $Z$ is reducible with $C = \Lambda \cup \Lambda'$ for some other line $\Lambda' \subseteq \bbP V$. We are going to show that $B \cap \Lambda'$ consists of three points. Suppose by contradiction $B \cap \Lambda'$ contains fewer than three points; since $B \subseteq C$, we deduce that $B \cap \Lambda$ contains at least two points. Notice $h_Z(3) = 1+2+2 +2 = 7 = \ell(Z)-1$: this guarantees $ \dim ( \linspan {v_3(A)} \cap \linspan{v_3(B)}) = 0$ therefore the two linear spans only intersect at $[F]$. However, 
\[
\ell(\Lambda \cap Z) = \ell(\Lambda \cap A)  + \ell(\Lambda \cap B) \geq  3+2 = 5;
\]
therefore $v_3(\Lambda \cap Z)$ is linearly dependent. This implies $ \linspan {v_3(A \cap \Lambda )} \cap \linspan{v_3(B \cap \Lambda)} \neq \emptyset$, in contradiction with the fact that the two spans only intersect at $[F]$, which is not an element of $\linspan{v_3(\Lambda)}$ because $F$ is concise. This shows that $B \cap \Lambda'$ contains three points.
\end{proof}

The case where $A$ and $B$ are not disjoint is characterized in the following result.
\begin{proposition}\label{prop: three collinear on plane}
 Let $F \in S^3 V$ be a concise form with $\dim V = 3$. Let $A,B$ be non-redundant decompositions of $F$, with $\ell(A) = \ell(B) = 4$. If $\ell(A \cap B) \geq 1$, then there is a line $\Lambda \subseteq \bbP V$ such that $A \setminus \Lambda = B \setminus \Lambda$ consists of exactly one point. In particular, both $A$ and $B$ have three collinear points on $\Lambda$.
 \end{proposition}
\begin{proof} 
Let $p = \ell(A \cap B)$. Let $F'$, $A',B'$ be constructed from $F,A,B$ as in \autoref{rmk: pass to disjoint}. Then $A'$ is a non-redundant decomposition of $F'$ of length $4-p \leq 3$ and $B'$ is a non-redundant decomposition of $F'$ of length at most $4$.

Suppose $A'$ is not contained in a line. In particular, $\ell(A') =3$ and $F'$ is concise by \autoref{cor: uniqueness for LI decomp}. Moreover $A'$ is the unique decomposition of length $3$. This contradicts the existence of $B'$.

Therefore there exists $V' \subseteq V$, with $\dim V' = 2$ and $A' \subseteq \Lambda = \bbP V'$. In particular $F' \in S^3 V'$ is not concise in $S^3 V$. If $F'$ is not concise in $S^3 V'$, then $F' = L^3$ for some $L \in V$; therefore $\{[L]\}$ is a decomposition of $F'$ and by \autoref{prop: sylvester's d+2}, any other non-redundant decomposition of $F'$ must have length at least $4$. On the other hand $\ell(B') \leq 3$, because if $\ell(B') = 4$, then $B = B'$, and $F'$ would be concise by \autoref{lemma: minimal plus one concise}. Therefore, the condition $\ell(B') \leq 3$ provides a contradiction with \autoref{prop: sylvester's d+2}. This guarantees $F'$ is concise in $S^3 V'$.

To conclude, we analyze two possibilities. If $\ell(A') = 2$, then $A'$ is the unique decomposition of $F'$ of length $2$ and $\ell(B') = 3$. In this case $\ell(A \cap B) = 2$ and it consists of one point on $\Lambda$ and one not in $\Lambda$, because $B'$ contains three of the four points of $B$ and $B' \subseteq \Lambda$. In particular $A \setminus \Lambda = B \setminus \Lambda = \{ [L]\}$ with $F = F' + L^3$. If $\ell(A') = 3 = \ell(B') = 3$, then $A \cap B = A \setminus \Lambda = B \setminus \Lambda = \{ [L]\}$ with $F = F' + L^3$. 
\end{proof}

\section{The general case}

In this section, we give the proof of \autoref{thm: main theorem}. This will follow from an induction argument, where the cases $n=1,2,3$ give the base of the induction. We provide first a result which analyzes the case $n=3$. 

If $\dim V= 4$, the generic Waring rank in $\bbP S^3 V$ is $5$. Sylvester's Pentahedral Theorem \cite{Sylv:PrinciplesCalculusForms,Cleb:TheorieFlachen} guarantees that a generic form in $S^3 V$ has a unique decomposition. We first study the situation of two disjoint decompositions, see also \cite{Ball:StratificationSigma5}.

\begin{proposition}\label{prop: surfaces on union of two lines}
Let $\dim V = 4$. Let $F \in S^3 V$ be a concise form. Let $A,B \subseteq \bbP V$ be two disjoint non-redundant decompositions of $F$ with $\ell(A),\ell(B) = 5$. Then $A \cup B$ is contained in the union of two lines.
\end{proposition}
\begin{proof}
 Since $A \cap B =\emptyset$, the union $Z=A\cup B$ is a set of length $10$. Since $F$ is concise, the decomposition $A$ is not contained in a hyperplane; therefore $Dh_A(0) = 1$, $Dh_A(1) = 3$. Since $A,B$ are non-redundant, by \autoref{prop: CB for nonredundant} $Z$ satisfies the Cayley-Bacharach property in degree $3$. Therefore $Dh_Z(4) + Dh_Z(3) \geq Dh_Z(0) + Dh_Z(1) \geq 4$. We deduce $Dh_Z(2)\leq 2$, which implies $Dh_Z(3), Dh_Z(4) \leq 2$ by \autoref{prop: basics HF}(iii). This guarantees $Dh_Z = (1,3,2,2,2)$.
 
 We apply \autoref{thm: BGM}: we have $Dh_Z(3) = Dh_Z(4) = 2 \leq 2$. Therefore there is a curve $C$ of degree $2$ such that, setting $Z' = Z \cap C$, one has $Dh_{Z'}(t) = Dh_{C} (t)$ for $t \leq 4$. Since $\deg(C) = 2$, $C$ is either a plane conic or a union of two lines. If $C$ is a plane conic, we have $Dh_{Z'}(t) = (1,2,2,2,2,\delta_5)$, with $\delta_{5} = 0,1$; in this case $Z' = Z \cap C$ contains at least $9$ points, therefore either $A \subseteq C$ or $B \subseteq C$, in contradiction with the conciseness of $F$, which implies $\langle A \rangle = \langle B \rangle = \bbP V$. Therefore $C$ is union of two lines, and we have $Dh_{Z'} = (1,3,2,2,2)$; this guarantees $Z = Z' \subseteq C$, as desired.
 \end{proof}

On the other hand, when $n \geq 4$, two distinct decompositions always intersect, as shown in the next result.
\begin{proposition} \label{prop: A and B must intersect} 
Let $n \geq 4$ and let $V$ be a vector space with $\dim V = n+1$. Let $F \in S^3 V$ be a concise form. Let $A$ be a non-redundant decomposition of $F$ of length $n+2$ and let $B$ be a non-redundant decomposition of $F$ of length $s \leq n+2$. Then $A \cap B \neq \emptyset$. 
\end{proposition}
\begin{proof}
Assume $B\cap A=\emptyset$, so that $Z=A\cup B$ is a set of length at most $2n+4$. Since $F$ is concise, then $A$ is not contained in a hyperplane; therefore $Dh_A(0) = 1$, $Dh_A(1) = n$. Since $A,B$ are non-redundant, $Z$ satisfies the Cayley-Bacharach property; therefore $Dh_Z(4) + Dh_Z(3) \geq Dh_Z(0) + Dh_Z(1) \geq n+1$. We deduce $Dh_Z(2)\leq 2$, which implies $Dh_Z(3), Dh_Z(4) \leq 2$; this provides a contradiction, for $n \geq 4$.
\end{proof}

The proof of \autoref{thm: main theorem} will follow because in the setting of \autoref{prop: A and B must intersect} one can guarantee that the intersection between two decompositions must be surprisingly large, similarly to \autoref{lemma: minimal plus one for fermat}.
\begin{proposition} \label{prop: main prop}
Let $V$ be a vector space with $\dim V = n+1$. Let $F \in S^3 V$ be a concise form. Let $A$ be a non-redundant decomposition of $F$ of length $n+2$ and let $B$ be a non-redundant decomposition of $F$ of length $s \leq n+2$. Then one of the following holds:
\begin{itemize}
 \item $\ell(A\cap B) \geq n-2$, and the points of $(A \setminus B) \cup (B \setminus A)$ are contained in the union of two planes;
 \item $\ell(A \cap B) \geq n-3$ and the points of $(A \setminus B) \cup (B \setminus A)$ are contained in the union of two lines.
\end{itemize}
\end{proposition}
 \begin{proof} 
 If $s = n+1$, then the statement follows immediately from \autoref{lemma: minimal plus one for fermat}. Therefore assume $s = n+2$. 
 
Let $A = \{ [L_1 ] \vvirg [L_{n+2}] \}$, $B = \{ [M_1 ] \vvirg [M_{n+2}] \}$. We proceed by induction on $n$. The case $n=1$ and $n=2$ are clearly verified. So assume $n \geq 3$. 

If $ n =3$ and $A \cap B = \emptyset$, then by \autoref{prop: surfaces on union of two lines} condition (ii) is verified. If $ n >3$, then $A \cap B \neq \emptyset$ by \autoref{prop: A and B must intersect}. Therefore assume $A \cap B \neq \emptyset$ with $[L_{n+2}] = [M_{n+2}] \in A\cap B$.

 Write, as usual,
 \begin{align*}
  F &= \alpha_1 L_1^3 + \cdots + \alpha_{n+1} L_{n+1}^3 + \alpha_{n+2} L_{n+2}^3 = \\
    &= \beta_1 M_1^3 + \cdots + \beta_{n+1} M_{n+1}^3 + \beta_{n+2} L_{n+2}^3.
 \end{align*}
Let $F' = F - \beta_{n+2} L_{n+2}^3$. Then $F'$ has a decomposition $B' = B \setminus \{ [L_{n+2}]\}$ of length $n+1$.

If $B'$ is linearly independent, by \autoref{cor: uniqueness for LI decomp}, we deduce that $F'$ is concise and $B'$ is its unique decomposition of length $n+1$. If $\alpha_{n+2} \neq \beta_{n+2}$, then $A$ defines a non-redundant decomposition of $F'$ of length $n+2$; in this case \autoref{lemma: minimal plus one for fermat} applies, and we deduce $\ell(A \cap B') \geq n-1$ and $A \setminus B'$ is collinear. Since $A \setminus B' = A \setminus B$, condition (ii) holds. If $\alpha_{n+2} = \beta_{n+2}$, then $A' = A \setminus \{[L_{n+2}]\}$ is a decomposition of $F'$ of length $n+1$; by Kruskal's criterion \autoref{cor: uniqueness for LI decomp}, we conclude $A' = B'$ which implies $A = B$ and the statement follows.

Now assume $B'$ is not linearly independent, which implies that $F'$ is not concise. Let $V' \subseteq V$ be the subspace in which $F'$ is concise. Since $F$ is concise, we have $\dim V' = \dim V - 1 = n$, and $\langle B' \rangle = V'$. Now, if $\alpha_{n+2} \neq \beta_{n+2}$, $A$ defines a non-redundant decomposition of $F'$ of length $n+2$. Moreover, $\langle A \rangle = V$. By \autoref{lemma: minimal plus one concise}, this would imply that $F'$ is concise, which is a contradiction. Therefore $\alpha_{n+2} = \beta_{n+2}$.

We reduced to the following setting. The form $F'$ is concise in $S^3 V'$, and $A',B'$ are two non-redundant decompositions of $F'$ of length $n+1$. Therefore the induction hypothesis applies. We deduce that condition (i) or condition (ii) hold for $A '$ and $B'$. Since $A = A' \cup\{ [L_{n+2}]\}$ and $B = B' \cup\{ [L_{n+2}]\}$, we deduce that condition (i) or condition (ii) holds for $A$ and $B$ as well.
\end{proof}

The proof of \autoref{thm: main theorem} is obtained directly from the results of this section.
\begin{proof}[Proof of \autoref{thm: main theorem}]
If $F$ has a unique decomposition $A$ of length $n+2$, then the Kruskal rank of $A$ must be at least $4$. Indeed, let $A' \subseteq A$ be a minimal linearly dependent set, so that the Kruskal rank of $A$ equals $\ell(A')-1$. Let $A = A' \cup A''$: $A''$ is linearly independent and $\langle v_3(A') \rangle \cap \langle v_3(A'') \rangle = \emptyset$. Correspondingly $F = F'+F''$, with $A'$ non-redundant decomposition of $F'$ and $A''$ non-redundant decomposition of $F''$. If $\ell(A') \leq 4$, then $F'$ has many decompositions of length $\ell(A')$, by the discussion on the case $n=2$. This implies that $A$ is not the unique decomposition of $F$, in contradiction with the assumption.

If $F$ does not have a unique decomposition, then either case (II) or case (III) occur, by \autoref{prop: main prop}.
\end{proof}

We record two easy consequence of \autoref{thm: main theorem}.
\begin{corollary}
Let $F \in S^3 V$ be a concise form admitting at least two distinct non-redundant decompositions of length $n+2$. Then the Kruskal rank of any non-redundant decomposition of length $n+2$ is at most $3$. Moreover, the family of decompositions of $F$ of length $n+2$ has dimension at least $2$. 
\end{corollary}
\begin{proof}
Since $F$ has at least two non-redundant decompositions of length $n+2$, either case (II) or case (III) of \autoref{thm: main theorem} is verified. This guarantees every non-redundant decomposition of $F$ of length $n+2$ has Kruskal rank at most $3$.

Write $A = \{ [L_0] \vvirg [L_{n+1}]\}$ for a non-redundant decomposition of length $n+2$ and assume $[L_0]\vvirg [L_3]$ are linearly dependent. Write $F = L_0^3 + \cdots + L_{n+1}^3$ and define $F' = L_0^3 + \cdots + L_3^3$. Then $F' \in S^3 V'$ with $\dim V' = 3$ and it has a non-redundant decomposition of length $4$. From the discussion on the case $n=2$, we deduce that $F'$ has a $2$-dimensional family of decomposition of length $4$. For every such decomposition $B'$, define $B = B' \cup \{ [L_4] \vvirg [L_{n+1}]\}$. Then $B$ is a non-redundant decomposition of $F$, showing that the family of decompositions of $F$ has dimension at least $2$.
\end{proof}

As a byproduct of these results, we completely characterize the genericity condition of Sylvester's Pentahedral Theorem. \autoref{thm: sylvester penta} can be regarded as an extension of \autoref{thm: n3} and of some of the results of \cite{BerGimIda:ComputingSymmetricRankSymmetricTensors}.
\begin{theorem}\label{thm: sylvester penta}
 Let $\dim V = 4$ and let $F \in S^3 V$ be a concise form. Then the following are equivalent:
 \begin{enumerate}[(i)]
  \item $F$ has at least two decompositions as sum of five powers of linear forms; 
  \item $F$ has a $2$-dimensional family of decompositions as sum of five powers of linear forms;
  \item $F = F' + L^3$ for $F' \in S^3 V'$ where $V' \subseteq V$ is a subspace of dimension $3$ and $L$ is a linear form with $L \notin V'$.
 \end{enumerate}
\end{theorem}
\begin{proof}
 \underline{$(i)\Rightarrow (iii)$.} Since the decomposition is not unique, either condition (I) or condition (II) of \autoref{thm: main theorem} hold. If condition (II) holds, then $F = F' + L^3$ as desired. If condition (I) holds, then $F = F'' + L_1^3 + L_2^3$ where $F'' \in S^3 V''$ for a subspace $V'' \subseteq V$ of dimension $2$; setting $F' = F'' + L_1^3$ we obtain the desired expression.
 
 \underline{$(iii) \Rightarrow (ii)$} Let $F = F' + L^3$. Every non-redundant decomposition of $F'$ provides a non-redundant decomposition of $F$. Since $\dim V'=3$, $F'$ has at least a $2$-dimensional family of decompositions.
 
 \underline{$(ii) \Rightarrow (i)$} This is trivially satisfied.
\end{proof}

\begin{corollary}
Let $\dim V = 4$.  A form $F \in S^3V$ has a unique decomposition as sum of five powers of linear forms if and only if 
\[
[ F ] \notin \bfJ( \Sub_3 , v_3(\bbP V))
\]
where $\Sub_3 $ denotes the variety of non-concise forms and $\bfJ(-,-)$ denotes the geometric join of two varieties.
\end{corollary}
The variety $\Sub_3$ is called \emph{subspace variety} in \cite[Ch. 3]{Lan:TensorBook} and it is defined by the vanishing of $4 \times 4$ minors of the first partial derivative map, see, e.g., \cite{Car:ReducingNumberVariables}. We refer to \cite[Ch. 8]{Harris:AlgGeo} for the definition and the basic properties of the join of two varieties.

\section{Characterization of the Terracini locus}\label{sec:terraciniloci}

In this section, we characterize a subvariety of the $(n+2)$-th Terracini locus of the third Veronese embedding of $\bbP^n$. Define the \emph{$r$-th concise Terracini locus} to be the subvariety of the Terracini locus arising from points in $\bbP^n$ whose span is the entire space. More precisely, for a vector space $V$ of dimension $n+1$, define
\[
 \bbT^{\con}_{r}(v_3(\bbP V)) =  \bar{\bbT_{r}(v_3(\bbP V)) \cap  \{ v_3(A) : A \in \bbP V^{(r)} \text{ and } \langle A \rangle = \bbP V \}}.
 \]
It is clear that if $ \bbT^{\con}_{r}(v_3(\bbP V)) $ is non-empty, then it is (union of) irreducible components of $\bbT_r(v_3(\bbP V))$. \autoref{thm: main theorem} allows us to describe $ \bbT^{\con}_{n+2}(v_3(\bbP V))$ as the closure of the set $v_3(K_3) \subseteq \bbP V^{(n+2)}$, where $K_3$ is the orbit, for the action of $\GL(V)$, described in \autoref{lemma: orbits}. In particular, this guarantees that it is irreducible. 
\begin{theorem}\label{thm: terracini locus characterization}
The concise Terracini locus $\bbT^{\con}_{n+2}(v_3(\bbP V))$ is the closure of 
\begin{equation}\label{eqn: concise terracini}
v_3(K_3) = \{ v_3(A) \in v_3(\bbP V) ^{(n+2)}: A \text{ concise with four coplanar points}\}.
\end{equation}
In particular, $\bbT^{\con}_{n+2}(v_3(\bbP V))$ is irreducible of dimension $n(n+1) + 2$.
\end{theorem}

The proof of \autoref{thm: terracini locus characterization} is based on \autoref{lemma: orbits} and two technical results presented below, \autoref{noint} and \autoref{noint2}.

\begin{lemma}\label{noint}
Let $d \geq 3$. Let $\bbP V_1 ,\bbP V_2 \subseteq \bbP V$ be disjoint linear spaces. For $i = 1,2$, let 
\[
H_i =  \linspan {T_{[M^d]} v_d(\bbP V) : M \in V_i }.
\]
Then $\bbP H_1 \cap \bbP H_2 = \emptyset$.
\end{lemma}
\begin{proof} We have $T_{[M^d]} v_d(\bbP V) = V \cdot M^{d-1}$. Therefore 
\[
H_i =  \linspan {T_{[M^d]} v_d(\bbP V) : M \in V_i \}} = V \cdot S^{d-1} V_i;
\]
since $V_1 \cap V_2 = 0$ and $d \geq 3$, we deduce $\bbP H_1 \cap \bbP H_2 = 0$.
\end{proof}

The next result is essentially contained in \cite[Lemma 5.10]{ChiCilib}.

\begin{lemma}\label{noint2} Let $V' \subseteq V$ be a proper subspace of $V$ and let $L_1,\dots L_k \in V'$. If the tangent spaces to $v_d(\bbP V)$ at $[L_1^d] \vvirg [L_k^d]$ are linearly dependent, so are the tangent spaces to $v_d(\bbP V')$ to $[L_1^d] \vvirg [L_k^d]$.
\end{lemma}
\begin{proof} If $[L_1^{d-1}] \vvirg [L_k^{d-1}]$ are linearly dependent then the statement is clearly satisfied. Therefore suppose they are independent. If $T_{[L_1^d]} v_d(\bbP V) \vvirg T_{[L_k^d]} v_d(\bbP V)$ are linearly dependent, then there exist $M_1 \vvirg M_k \in V$, not all zero, such that 
\[
 M_1 L_1^{d-1} + \cdots + M_k L_k^{d-1} = 0.
\]
In fact, we show $M_j \in V'$ for every $j$. To see this, let $V''$ be a complement to $V'$, so that $V = V' \oplus V''$; write $M_j = M_j' + M_j''$ for some $M_j' \in V'$, $M_j'' \in V''$. We have 
\begin{align*}
& M_1 L_1^{d-1} + \cdots + M_k L_k^{d-1} = \\ & (M_1' L_1^{d-1} + \cdots + M_k' L_k^{d-1}) + (M_1'' L_1^{d-1} + \cdots + M_k'' L_k^{d-1})  = 0.
\end{align*}
Since $L_j \in V'$, the two summands in the expression above are elements of $S^{d} V'$ and $V'' \cdot S^{d-1} V'$ respectively; since these two subspaces are linearly independent, each summand must vanish. Since $V' \cap V'' = 0$, the elements of $V'' \cdot S^{d-1} V'$ can be regarded as elements of $V'' \otimes S^{d-1} V' \subseteq S^d V$; in particular, we have $M_1'' \otimes  L_1^{d-1} + \cdots + M_k'' \otimes L_k^{d-1} = 0$. If some $M_j''$ is nonzero, we deduce that $L_1^{d-1} \vvirg L_k^{d-1}$ are linearly dependent, which contradicts the initial assumption. Therefore $M_j'' = 0$ for every $j$ and $M_j = M_j'$. 
\end{proof}

We can now provide a complete proof of \autoref{thm: terracini locus characterization}.
\begin{proof}[Proof of \autoref{thm: terracini locus characterization}]
Let $\Omega \subseteq \bbP V^{(n+2)}$ be the subset of sets of points spanning the entire $\bbP V$. By definition of concise Terracini locus, $v_3(\Omega) \cap \bbT^{\con}_{n+2}(v_3(\bbP V))$ is Zariski open in $\bbT^{\con}_{n+2}(v_3(\bbP V))$ and non-empty if $\bbT^{\con}_{n+2}(v_3(\bbP V))$ is non-empty. We prove that $v_3(\Omega) \cap \bbT^{\con}_{n+2}(v_3(\bbP V)) = v_3(K_3)$.

Let $A \in \Omega$ and write $A = \{ [L_1] \vvirg [L_{n+2}]\}$. Let $r$ be the Kruskal rank of $A$: by \autoref{lemma: orbits}, $A \in K_r$ and it can be normalized as $A = \{[x_0] \vvirg [x_n],[x_0 + \cdots + x_{r-1}]\}$. Write $T_j = T_{[x_j]}v_3(\bbP V)$, $T_+ = T_{[x_0 + \cdots + x_{r-1}]} v_3(\bbP V)$; let $V' = \langle x_0 \vvirg x_{r-1}\rangle$ and let $T_j' = T_{[x_j]}v_3(\bbP V')$, $T_+' = T_{[x_0 + \cdots + x_{r-1}]} v_3(\bbP V')$.

If $v_3(A)$ is an element of the Terracini locus, then $T_0 \vvirg T_n , T_+$ are linearly dependent. By \autoref{noint}, the same must hold for the tangent spaces $T_0 \vvirg T_{r-1}, T_+$. Further, by \autoref{noint2}, the same must hold for $T'_0 \vvirg T_{r-1}',T_+'$. In particular, it is enough to $A' \subseteq A$ defined by $A' = \{ [x_0] \vvirg [x_{r-1}], [x_0 + \cdots +x_{r-1}]\} \in (\bbP V')^{(r+2)}$.

We deduce that $v_3(A')$ belongs to the concise Terracini locus $\bbT^{\con}_{r+2}(v_3(\bbP V'))$. By the action of $\GL(V')$, the same holds for every set of $r+2$ points in linear general position in $\bbP V'$. Passing to the closure, we obtain $v_3(\bbP V')^{(r+2)}$ is entirely contained in the Terracini locus. By Terracini's Lemma, we deduce that all fibers of $\pi_\sigma : A\sigma_{r+2}(v_3(\bbP V')) \to \bbP S^3 V'$ have positive dimension. \autoref{thm: main theorem} guarantees that this is the case only if $r \leq 3$. This shows that $A \subseteq K_3$, or equivalently that it has four coplanar points. We obtained $v_3(\Omega) \cap \bbT^{\con}_{n+2}(v_3(\bbP V)) \subseteq v_3(K_3)$. 

The reverse inclusion holds because of \autoref{thm: main theorem} and this concludes the proof of the equality in \eqref{eqn: concise terracini}.

Finally, by \autoref{lemma: orbits}, we conclude that $\bbT^{\con}_{n+2}(v_3(\bbP^n)) = \bar{v_3(K_3)}$ is irreducible, and of dimension $n(n+1) + 2$.
\end{proof}
\bibliographystyle{alphaurl}
\bibliography{bibKruskal.bib}

\end{document}